\documentclass{fundam}

\publyear{25}
\papernumber{2}
\volume{194}
\issue{3}
\theDOI{10.46298/fi.10201}

\versionForARXIV

\usepackage{amsmath,amsfonts,amssymb}
\usepackage[mathscr]{euscript}
\usepackage{amsmath}
\usepackage{graphicx,latexsym}
\usepackage{lscape}
\usepackage{fixmath}
\usepackage{multicol}
\usepackage{graphicx}
\usepackage{caption}
\usepackage{float}
\usepackage{cite}
\usepackage{subfig}
\usepackage{setspace}
\usepackage{xcolor}
\usepackage{multirow}
\usepackage{hhline} 
\usepackage[utf8]{inputenc}
\usepackage{gensymb}
\usepackage{setspace}
\usepackage{abstract}
\usepackage{caption}
\usepackage{ulem}

\newtheorem{observation}[definition]{Observation}
\newtheorem{property}[definition]{Property}
\newcommand{\red}{\color{red}}
\newcommand{\blue}{\color{blue}}

\newcommand{\gcover}{{\rm gcover}}

\newcommand{\gcovere}{{\rm gcover_{e}}}

\newcommand{\gparte}{{\rm gpart_{e}}}
\newcommand{\BF}{{\rm BF}}

\newcommand{\diam}{{\rm diam}}

\begin{document}
\title{The geodesic cover problem for butterfly networks}

\author{
\!Paul Manuel \\
Department of Information Science, College of Computing Science and Engineering, \\
Kuwait University, Kuwait \\
\href{mailto:pauldmanuel@gmail.com}{pauldmanuel@gmail.com},~\href{mailto:p.manuel@ku.edu.kw}{p.manuel@ku.edu.kw}
\and
Sandi Klav\v zar \\
Faculty of Mathematics and Physics, University of Ljubljana, Slovenia \\
Faculty of Natural Sciences and Mathematics, University of Maribor, Slovenia \\
Institute of Mathematics, Physics and Mechanics, Ljubljana, Slovenia \\
\href{mailto:sandi.klavzar@fmf.uni-lj.si}{sandi.klavzar@fmf.uni-lj.si}
\and
R. Prabha \\
Department of Mathematics, Ethiraj College for Women, Chennai, Tamilnadu, India\\
\href{mailto:prabha75@gmail.com}{prabha75@gmail.com}
\and
Andrew Arokiaraj \\
Department of Mathematics, School of Science and Humanities, \\
Shiv Nadar University Chennai, Kalavakkam, India \\
\href{mailto:andrew23610032@snuchennai.edu.in}{andrew23610032@snuchennai.edu.in}
}


\maketitle
%
%
%
%
%

\runninghead{P. Manuel et al.}{The geodesic cover problem for butterfly networks}

\vspace{-5ex}

\begin{abstract}
A geodesic cover, also known as an isometric path cover, of a graph is a set of geodesics which cover the vertex set of the graph. An edge geodesic cover of a graph is a set of geodesics which cover the edge set of the graph. 
The geodesic (edge) cover number of a graph is the cardinality of a minimum (edge) geodesic cover.
The (edge) geodesic cover problem of a graph is to find the (edge) geodesic cover number of the graph. 
Surprisingly, only partial solutions for these problems are available for most situations. In this paper we demonstrate that the geodesic cover number of the $r$-dimensional butterfly is $\lceil (2/3)2^r\rceil$ and that its edge geodesic cover number is $2^r$. 
\end{abstract}

\noindent{\bf Keywords}: isometric path; geodesic cover; edge geodesic cover; bipartite graph, butterfly network 

\medskip
\noindent{\bf AMS Subj.\ Class.~(2020)}: 05C12, 05C70

\section{Introduction}

Let $G = (V(G), E(G))$ be a graph. The {\it distance} $d(x, y)$ between vertices $x,y\in V(G)$ is the length of a shortest $x, y$-path in $G$; any such path is called a {\it geodesic}. The {\it diameter} $\diam(G)$ of $G$ is the maximum distance between any pair of vertices in $G$, that is,  
$\diam(G) = \max_{u,v} d_G(u, v)$. A subgraph $H$ of $G$ is {\it isometric} if $d_G(x, y) = d_H(x, y)$ for all $x, y\in V(H)$. 

A \textit{geodesic cover} of a graph $G$ is a set $S$ of geodesics such that each vertex of $G$ belongs to at least one geodesic of $S$. It is popularly known as \textit{isometric path cover}~\cite{Fitz1999, Fitz2001, Manuel18, Manuel19, pan-2006}. The geodesic cover problem is one of the fundamental problems in graph theory. The concept of geodesic cover is widely used in social networks, computer networks, and fixed interconnection networks \cite{Manuel18}. 
Throughout this paper, $Z(G)$ denotes the set of all geodesics of $G$ and $M(G)$ denotes the set of all maximal (with respect to inclusion) geodesics of $G$.  

Given $Y \subseteq Z(G)$ and $S \subseteq V(G)$,  a \textit{geodesic cover} of the triple $(Y,S, G)$ is a set of geodesics of $Y$ that cover $S$. 
Given $Y \subseteq Z(G)$ and $S \subseteq V(G)$, the \textit{geodesic cover number} of $(Y,S, G)$, $\gcover(Y,S,G)$, is the minimum number of geodesics of $Y$ that cover $S$. Note that there exist situations where $\gcover(Y,S,G)$ may not exist. 

When $Y\subseteq Z(G)$ and $S=V$, $\gcover(Y,V,G)$ is denoted by $\gcover(Y,G)$.

When $Y=Z(G)$ and $S\subseteq V$,  $\gcover(Z(G),S,G)$ is denoted by $\gcover(S,G)$.

When $Y=Z(G)$ and $S = V$,  $\gcover(Z(G),V,G)$ is denoted by $\gcover(G)$.

\noindent
Given $Y \subseteq Z(G)$ and $S \subseteq V(G)$, the \textit{geodesic cover problem} of $(Y,S)$ is to find $\gcover(Y,S,G)$ of $G$. The \textit{geodesic cover problem} of $G$ is to find $\gcover(G)$ of $G$.
An \textit{edge geodesic cover} of a graph is a set of geodesics which cover the edge set of the graph. 
The \textit{edge geodesic cover number} of a graph $G$, $\gcovere(G)$, is the cardinality of a minimum edge geodesic cover.
The \textit{edge geodesic cover problem} of a graph $G$ is to find $\gcovere(G)$.

The geodesic cover problem is known to be NP-complete~\cite{ChDa22,LiSa22}. Apollonio  et al. \cite{ApCaSi04} have studied induced path covering problems in grids. Fisher and Fitzpatrick \cite{Fish2001} have shown that the geodesic cover number of the ($r\times r$)-dimensional grid is $\lceil 2r/3\rceil$.  
The  geodesic  cover  number  of  the  $(r\times s)$-dimensional  grid  is $s$ when $r\geq s(s-1)$, cf.~\cite{Manuel19}. On the other hand, the complete solution of the geodesic cover problem for  the two-dimensional grid is still unknown, cf.~\cite{Manuel19}.  There is no literature for the geodesic cover problem on multi-dimensional grids. 
 
The geodesic cover problems for cylinder and $r$-dimensional grids are discussed in \cite{Manuel19}. In particular, the isometric path cover number of the $(r\times r)$-dimensional torus is $r$ when $r$ is even, and is either $r$ or $r+ 1$ when $r$ is odd. In~\cite{PaCh05}, the geodesic cover problem was studied on block graphs, while in~\cite{pan-2006} it was investigated on complete $r$-partite graphs and Cartesian products of two or three complete graphs.   
  
Fitzpatrick et al.~\cite{Fitz1999, Fitz2001} have shown that the geodesic cover number of the hypercube $Q_r$ is at least $2r/(r+1)$ and they have provided a partial solution when $r+1$ is a power of $2$. The complete solution for the geodesic cover number of hypercubes is also not yet known, cf.~\cite{Fitz1999, Fitz2001, Manuel18}.  Manuel \cite{Manuel19} has proved that the geodesic cover number of the $r$-dimensional Benes network is $2^r$. In~\cite{Manuel18,Manuel19} the (edge) geodesic cover problem of butterfly networks was stated as an open problem. In this paper we solve these two problems. 

\section{Preliminaries}
\label{sec:preliminaries}

The results discussed in this section will be used as tools to prove the key results of this paper. 

\begin{lemma}
	\label{lem:gcover-max-geo}
	If $G$ is a connected graph, then the following hold. 
	\begin{enumerate} 
		\item[(i)]
		If $S' \subseteq S'' \subseteq V(G)$ and $Y\subseteq Z(G)$, then $\gcover(Y, S'', G)$ $\geq$ $\gcover(Y, S', G)$.
		\item[(ii)] 
		If $Y' \subseteq Y''\subseteq Z(G)$ and $S\subseteq V(G)$, then $\gcover(Y'', S, G) \leq \gcover(Y', S, G)$.
		\item[(iii)] 	
		$\gcover(G)$ $=$ $\gcover(M(G),G)$.
	\end{enumerate}
\end{lemma}
\begin{proof}
Assertions (i) and (ii) are straightforward, hence we consider only (iii). Since $M(G) \subseteq Z(G)$,  (ii) implies \[\gcover(G) = \gcover(Z(G), V, G) \leq \gcover(M(G), V, G) = \gcover(M(G), G).\] Since each geodesic is a subpath of some maximal geodesic, for each geodesic cover $S$ of $Z(G)$, there exists a geodesic cover $S'$ of $M(G)$ such that $|S| = |S'|$. Therefore, \[\gcover(Z(G), V, G) \geq \gcover(M(G), V, G)\] and consecutively \[\gcover(G) = \gcover(Z(G), V, G) \geq \gcover(M(G), V, G) = \gcover(M(G), G).\]
\end{proof}

\begin{proposition}
	\label{prop:gcover-com-bipartite}
	If $K_{r,r}$, $r \ge 2$, is a complete bipartite graph, then $\gcover(K_{r,r}) = \lceil(2/3)r\rceil$. 
\end{proposition}

\begin{proof}
Clearly, each maximal geodesic of $K_{r,r}$ is a (diametral) path of length $2$. Therefore, $\gcover(K_{r,r}) \geq \lceil(2/3)r\rceil$. On the other hand, it is a simple exercise to construct a geodesic cover of cardinality $\lceil(2/3)r\rceil$. 
\end{proof}

Butterfly is considered as one of the best parallel architectures~\cite{HsuLin08, Leighton1992, SuRaRaRa-2021}. For $r\ge 3$, the $r$-dimensional {\it butterfly network} $BF(r)$ has vertices $[j, s]$, where $s\in \{0,1\}^r$ and $j\in \{0,1,\ldots, r\}$. The vertices $[j, s]$ and $[j', s ']$ are adjacent if $j'=j+1$, and either $s = s'$ or $s$ and $s'$ differ precisely in the $j^{\rm th}$ bit. $BF(r)$ has $(r + 1)2^r$ vertices and $r2^{r+1}$ edges. A vertex $[j, s]$ is at \textit{level} $j$ and {\it row} $s$. There are two standard graphical representations for $\BF(r)$, normal representation and diamond representation, see Fig.~\ref{fig:Butterfly-diamond-normal}. 

\begin{figure}[ht!]
	\centering
	\includegraphics[scale=0.42]{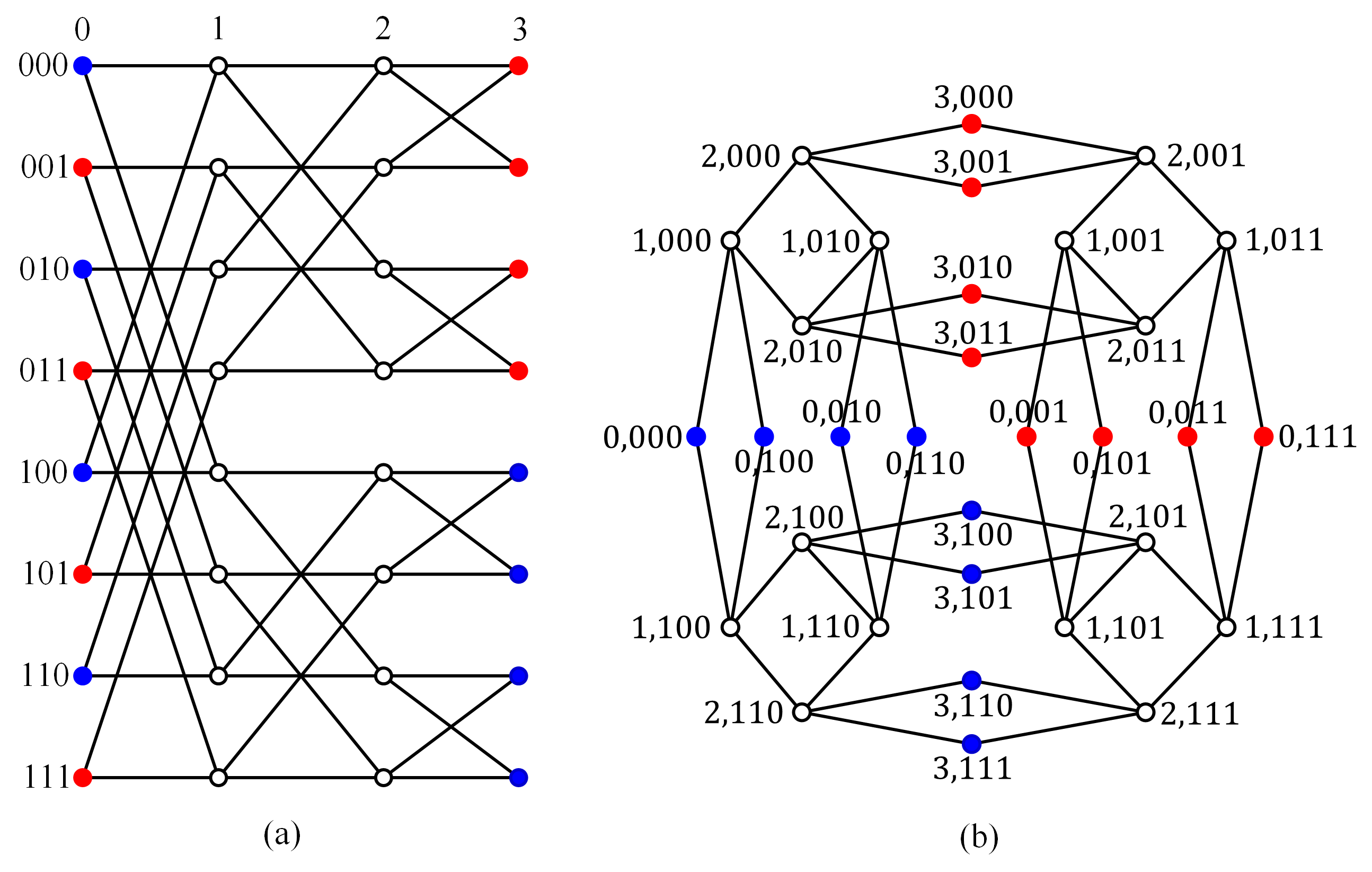}
	\caption{(a) Normal representation of $\BF(3)$  (b) Diamond representation of $\BF(3)$.}
	\label{fig:Butterfly-diamond-normal}
\end{figure}

Estimating the lower bound of $\gcover(\BF(r))$ in Section~\ref{subsec:lowerbound-gcover-butterfly}, we will use the diamond representation of $\BF(r)$, while estimating the upper bound of $\gcover(\BF(r))$ in Section \ref{subsec:upperbound-gcover-butterfly}, the normal representation of $\BF(r)$ will be used.

\begin{lemma}
	\label{lem:butterfly-geo0}
	A geodesic of $\BF(r)$ contains at most two vertices of level $0$ and at most two  vertices of level $r$. Moreover, if a geodesic contains two vertices of level $0$, then they are the ends of the geodesic (and similarly for level r).
\end{lemma}

\begin{proof}
Let us assume that there exists a geodesic $P$ which contains more than two vertices of level $0$, say $v_{i}, v_{j}$, and $v_{k}$. See Fig.~\ref{fig:Butterfly-4-dim1}(b). Then one of these three vertices must be an internal vertex of $P$, say $v_j$. 	The deletion of the vertices at level $0$ from $\BF(r)$ disconnects $\BF(r)$ into two vertex disjoint components $G_1$ and $G_2$, where both $G_1$ and $G_2$ are isomorphic to $\BF(r-1)$, cf.~\cite{Leighton1992,RaMaPa16, ToDa05}. Since $v_{j}$ is an internal vertex of $P$ and of degree $2$, its neighbors $v_{j-1}$ and $v_{j+1}$ also lie in $P$. Moreover, one of the adjacent vertices $v_{j-1}$, $v_{j+1}$ lie in $G_1$ and the other lie in $G_2$. Also, $v_i$ has two adjacent vertices, say $v_{i-1}\in V(G_1)$ and $v_{i+1}\in V(G_2)$. Since $G_1$ and $G_2$ are isomorphic, the $v_{i-1},v_{j-1}$-geodesic and the $v_{i+1},v_{j+1}$-geodesic are isomorphic, which in turn implies that $d(v_{i},v_{j-1}) = d(v_{i},v_{j+1})$. This is not possible as $P$ is a geodesic.
\end{proof}

As the butterfly network is symmetrical  with respect to level $0$, it is symmetrical with respect to level $r$, cf.~\cite{Leighton1992, RaMaPa16, ToDa05}. Using the logic of the proof of Lemma~\ref{lem:butterfly-geo0}, one can prove the following.

\begin{corollary}
	\label{cor:butterfly-geo}
	If both end vertices of a geodesic $P$ of $\BF(r)$ are either at level $0$ or at level $r$, then $P$ is maximal.  
\end{corollary}

\begin{lemma}
	\label{lem:butterfly-geo2}
	A geodesic of $\BF(r)$ covers at most three vertices of degree $2$.  
\end{lemma}

\begin{proof}
Suppose a geodesic $P$ contains four vertices $a,b,c,d$ of degree $2$. Then all the vertices $a,b,c,d$ are at level $0$ or at level $r$. By Lemma \ref{lem:butterfly-geo0}, three of these vertices can not be at the same level. Assume without loss of generality that $a$ and $b$ are at level $0$ and $c$ and $d$ are at level $r$. Let $P(a,b)$ be a subpath of $P$ between $a$ and $b$, and $P(c,d)$ the subpath of $P$ between $c$ and $d$. By Corollary~\ref{cor:butterfly-geo}, $P(a,b)$ and $P(c,d)$ are maximal geodesics, a contradiction. 
\end{proof}

\begin{corollary}
\label{cor:butterfly-geo1}
If a geodesic $P$ of $\BF(r)$ covers three vertices of degree $2$, then the end vertices of $P$ are of degree $2$, and $P$ is maximal.
\end{corollary}

\section{The geodesic cover problem for $\BF(r)$}
\label{sec:gcover-butterfly}

\subsection{A lower bound for $\gcover(\BF(r))$}
\label{subsec:lowerbound-gcover-butterfly}
\subsubsection{Revisiting properties of $\BF(r)$}
\label{subsec:properies-butterfly}
In this section, we use the following notations.
Let $U$ and $W$ denote the sets of vertices at  level $0$ and  level $r$ in $\BF(r)$, respectively. Further, let $U = U^{b} \bigcup U^{r}$, where $U^{b}= \{u^{b}_1, u^{b}_2,\dots, u^{b}_{2^{r-1}}\}$ and $U^{r}=  \{u^{r}_1, u^{r}_2,\dots, u^{r}_{2^{r-1}}\}$. Similarly, $W = W^{b} \bigcup W^{r}$ where
$W^{b}= \{w^{b}_1, w^{b}_2,\dots, w^{b}_{2^{r-1}}\}$ and $W^{r}=  \{w^{r}_1, w^{r}_2,\dots, w^{r}_{2^{r-1}}\}$, see Fig.~\ref{fig:Butterfly-4-dim1}(a).

\begin{figure}[ht!]
	\centering
	\includegraphics[scale=0.40]{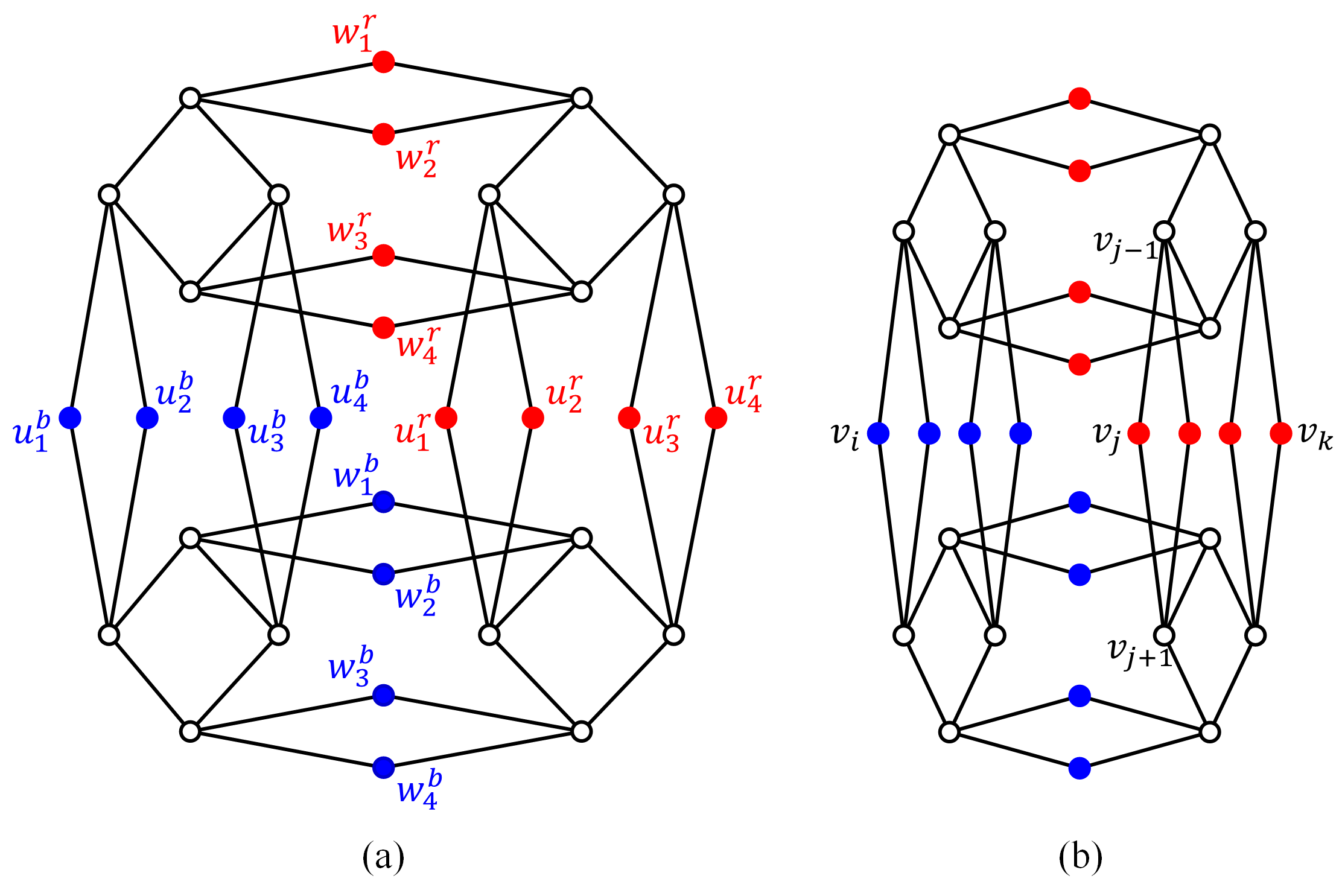}
	\caption{(a) Vertices from $U^{b}$ and $W^{b}$ are blue, vertices from $U^{r}$ and $W^{r}$ are red. (b) $v_{j}$ is an internal vertex of $P$ and of degree $2$. Its neighbors $v_{j-1}$ and $v_{j+1}$ lie in $P$. One of the adjacent vertices $\{v_{j-1}$, $v_{j+1}\}$ lie in $G_1$ and the other lie in $G_2$.}
	\label{fig:Butterfly-4-dim1}
\end{figure}

In order to gain an in-depth understanding  of the behavior of the geodesics of $\BF(r)$, it is necessary to enumerate all the maximal geodesics of $\BF(r)$,  cf.~\cite{HsuLin08,Manuel2008}.

\begin{lemma}
	\label{lem:max-geo-butterfly}
The following facts hold in $\BF(r)$. 
	\begin{enumerate}
	\item
	For $u^{b}_{i},u^{b}_{j} \in U^b$, a maximal geodesic $P(u^{b}_{i},u^{b}_{j})$ between $u^{b}_{i}$ and $u^{b}_{j}$ does not intersect $W$. For  $u^{r}_{i},u^{r}_{j} \in U^r$, a maximal geodesic $P(u^{r}_{i},u^{r}_{j})$ between $u^{r}_{i}$ and $u^{r}_{j}$ does not intersect $W$. 
	\item
	For $w^{b}_{i},w^{b}_{j} \in W^b$, a maximal geodesic $P(w^{b}_{i},w^{b}_{j})$ between $w^{b}_{i}$ and $w^{b}_{j}$ does not intersect $U$. For  $w^{r}_{i},w^{r}_{j} \in W^r$, a maximal geodesic $P(w^{r}_{i},w^{r}_{j})$ between $w^{r}_{i}$ and $w^{r}_{j}$ does not intersect $U$. 
	\item
	If $u^b \in U^b$, $u^r \in U^r$, and $w \in W$,  then there is a unique geodesic $P_w(u^b, u^r)$ between $u^b$ and $u^r$ passing through $w$. This geodesic is the concatenation of geodesics $P(u^b,w)$ and $P(w,u^r)$,  where $u^b \in U^b$, $u^r \in U^r$,  and $w \in W$. Consequently, if $u^b \in U^b$ and $u^r \in U^r$, then there are $2^r$ maximal $u^b, u^r$-geodesics.
		
	If $w^b \in W^b$, $w^r \in W^r$, and $u \in U$, then there is a unique geodesic $P_u(w^b, w^r)$ between $w^b$ and $w^r$ passing through $u$. This geodesic is the concatenation of geodesics $P(w^b,u)$ and $P(u,w^r)$,  where $w^b \in W^b$, $w^r \in W^r$,  and $u \in U$. Hence, if $w^b \in W^b$ and $w^r \in W^r$, then there are $2^r$ maximal $w^b, w^r$-geodesics.
	\end{enumerate}
\end{lemma}

\begin{proof} 
$\BF(r) - W$ consists of two components both isomorphic to $\BF(r-1)$, cf.~\cite{Leighton1992, HsuLin08}. As these components are furthermore convex in $\BF(r)$, we get that if either $u^{b}_{i},u^{b}_{j} \in U^b$,  a maximal $u^{b}_{i},u^{b}_{j}$-geodesic does not intersect $W$. Analogously the other assertions hold. The assertion (3) follows from the fact that when $u \in U$ and $w \in W$,	a $u,v$-geodesic is unique.
\end{proof}

Set now 
$$M_{U,W}(\BF(r)) = \{P:\ P\ {\rm is \ a \ maximal}\ x,y{\mbox -}{\rm geodesic,\ either}\ x,y \in U\ {\rm or}\ x,y \in W\}\,.$$
By Lemma~\ref{lem:max-geo-butterfly}, the set of geodesics $M_{U,W}(\BF(r))$ is partitioned into six disjoint subsets as follows. 
\begin{observation}
\label{obs:M-buttefly}
$M_{U,W}(\BF(r))$ partitions into the following sets: 
	\begin{enumerate}
	\item[(i)] 
		$\{P(u^{b}_{i},u^{b}_{j}): u^{b}_{i},u^{b}_{j} \in U^b\}$,
	\item[(ii)] 
		$\{P(u^{r}_{i},u^{r}_{j}): u^{r}_{i},u^{r}_{j} \in U^r\}$,
	\item[(iii)] 
		$\{P(w^{b}_{i},w^{b}_{j}): w^{b}_{i},w^{b}_{j} \in W^b\}$,
	\item[(iv)] 
		$\{P(w^{r}_{i},w^{r}_{j}): w^{r}_{i},w^{r}_{j} \in W^r\}$,
	\item[(v)] 
		$\{P_w(u^b, u^r) : u^b \in U^b, u^r \in U^r, w \in W\}$,
	\item[(vi)]
		$\{P_u(w^b, w^r) : w^b \in W^b, w^r \in W^r, u \in U\}$.
	\end{enumerate}
\end{observation}

$M_{U,W}(\BF(r))$ is thus the set of all maximal $x,y$-geodesic, where either $x,y \in U$ or $x,y \in W$. In (\textit{i})-(\textit{iv}) of Observation~\ref{obs:M-buttefly}, given a pair of vertices $x,y \in U$ or $x,y \in W$, there are more than one maximal geodesics between $x$ and $y$ in $M_{U,W}(\BF(r))$. 
Now we  define $M'_{U,W}(\BF(r)) \subset M_{U,W}(\BF(r))$ as follows. First, $M'_{U,W}(\BF(r))$ contains all the geodesics from (\textit{v}) and (\textit{vi}) of Observation~\ref{obs:M-buttefly}. 
Second, for each pair of vertices $u^{b}_{i},u^{b}_{j} \in U^b$ from (\textit{i}),  $u^{r}_{i},u^{r}_{j} \in U^r$ from (\textit{ii}), $w^{b}_{i},w^{b}_{j} \in W^b$ from (\textit{iii}), and $w^{r}_{i},w^{r}_{j} \in W^r$ from (\textit{iv}),
select an arbitrary but fixed geodesic between them and add it to $M'_{U,W}(\BF(r))$. In this way the set $M'_{U,W}(\BF(r))$ is defined. 

For the sake of clarity, we write the members of $M'_{U,W}(\BF(r))$ explicitly below:
\begin{align*} 
M'_{U,W}(\BF(r)) = &
\{P'(u^{b}_{i},u^{b}_{j}): P'(u^{b}_{i},u^{b}_{j}) \ {\rm is \ a \ fixed \ geodesic \ between} \ u^{b}_{i},u^{b}_{j} \in U^b\} \\ 
& \cup 	
\{P'(u^{r}_{i},u^{r}_{j}): P'(u^{r}_{i},u^{r}_{j}) \ {\rm is \ a \ fixed \ geodesic \ between} \ u^{r}_{i},u^{r}_{j} \in U^r\} \\
& \cup 
\{P'(w^{b}_{i},w^{b}_{j}): P'(w^{b}_{i},w^{b}_{j}) \ {\rm is \ a \ fixed \ geodesic \ between} \ w^{b}_{i},w^{b}_{j} \in W^b\} \\
& \cup
\{P'(w^{r}_{i},w^{r}_{j}): P'(w^{r}_{i},w^{r}_{j}) \ {\rm is \ a \ fixed \ geodesic \ between} \ w^{r}_{i},w^{r}_{j} \in W^r\} \\
& \cup 
\{P_w(u^b, u^r) : u^b \in U^b, u^r \in U^r, w \in W\} \\
& \cup
\{P_u(w^b, w^r) : w^b \in W^b, w^r \in W^r, u \in U\}.
\end{align*}
Note that for each pair $u_i, u_j$, for each pair $w_i, w_j$, for each triple $u^b, u^r,w$, and for
each triple $w^b, w^r,u$, the set $M'_{U,W}(\BF(r))$  contains a unique corresponding geodesic. Note also that \[M'_{U,W}(\BF(r)) \subset M_{U,W}(\BF(r)) \subset M(\BF(r)).\]
\subsubsection{Estimating a lower bound for $\gcover(\BF(r))$}
\label{sec:lower-bnd-gcover}

\begin{lemma}
\label{lem:butterfly-geo4}
 If $U$ and $V$ are as above, then 
$$\gcover(BF(r))  \geq  \gcover [M'_{U,W}(BF(r)), U\cup W, BF(r)]\,. $$
\end{lemma}

\begin{proof}
Set $G = \BF(r)$. To prove the lemma, we are going to show that 
\begin{align*} 
\gcover(G)  & \geq \gcover[M(G), U\cup W, G] \\
& = \gcover[M_{U,W}(G), U\cup W, G] \\
& = \gcover[M'_{U,W}(G), U\cup W, G]\,. 		
\end{align*}
	By Lemma~\ref{lem:gcover-max-geo}, we get $\gcover(G)  \geq \gcover[M(G), U\cup W, G]$.  By Observation~\ref{obs:M-buttefly} and the definition of $M'_{U,W}(G)$, $\gcover[M_{U,W}(G), U\cup W, G] = \gcover[M'_{U,W}(G), U\cup W, G]$. 	Next we prove that $ \gcover[M(G), U\cup W, G] = \gcover[M_{U,W}(G), U\cup W, G]$.
	
 Since $M_{U,W}(G) \subseteq M(G)$, by Lemma~\ref{lem:gcover-max-geo}, we get \[\gcover[M(G), U\cup W, G] \leq \gcover[M_{U,W}(G), U\cup W, G].\] 	Now it is enough to prove that  \[\gcover[M(G),U\cup W, G] \geq \gcover[M_{U,W}(G),U\cup W, G].\]
	By Lemma~\ref{lem:butterfly-geo2}, a geodesic covers at most three vertices of $U \cup W$. If $P$ is a member of $M(G)$ such that $P$ covers three vertices of $U \cup W$ in $\BF(r)$, then by Corollary~\ref{cor:butterfly-geo1}, $P \in M_{U,W}(G)$. On the other hand, if $P$ is a member of $M(G)$ covering two vertices $v_1$ and $v_2$ of $U \cup W$, then by Observation~\ref{obs:M-buttefly} there exists a geodesic $Q$ of $M_{U,W}(G)$ such that $Q$ covers both vertices $v_1$ and $v_2$. Hence, $\gcover[M(G),U\cup W, G] \geq \gcover[M_{U,W}(G),U\cup W, G]$.	
\end{proof}

Let us consider two sets $X$ and $Y$ where $X= X^{b} \cup X^{r}$,  $Y= Y^{b} \cup Y^{r}$,  	
$X^{b} = \{x^b_{1}, x^b_{2},\dots, x^b_{2^{r-1}}\}$,
$X^{r} = \{x^r_{1}, x^r_{2},\dots, x^r_{2^{r-1}}\}$,
$Y^{b} = \{y^b_{1}, y^b_{2},\dots, y^b_{2^{r-1}}\}$,
and
$Y^{r} = \{y^r_{1}, y^r_{2},\dots, y^r_{2^{r-1}}\}$. Now we define a complete bipartite graph $G'$ with the bipartition $X, Y$. Let us further define another complete bipartite graph $G''$ with the bipartition $X_0 = X \cup  \{x_0\}, Y_0 = Y \cup \{y_0\}$. The graphs $G'$ and $G''$ are presented in Fig.~\ref{fig:comp-bipartite-XY}. 

\begin{figure}[ht!]
	\centering
	\includegraphics[scale=0.42]{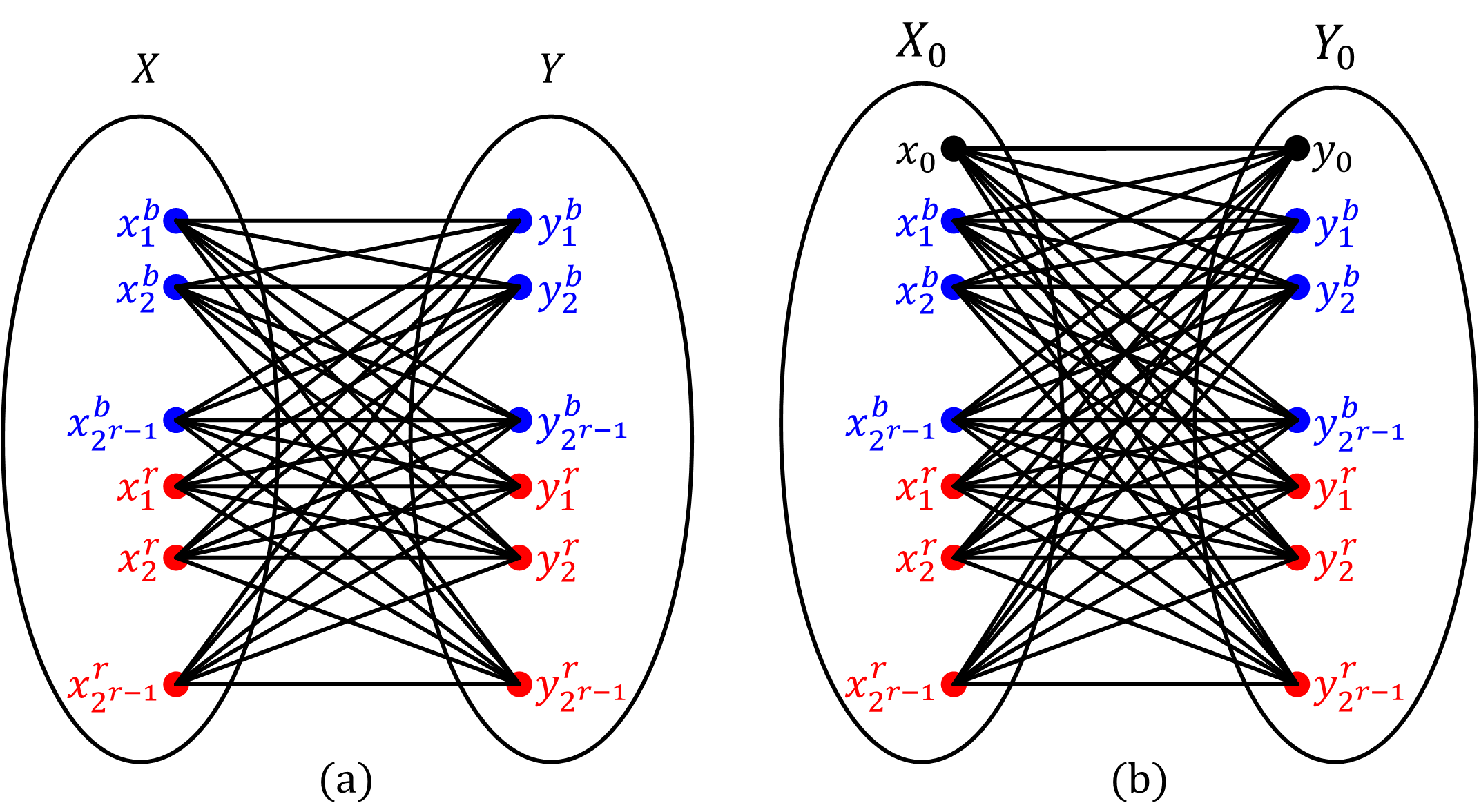}
	\caption{
	(a)  The complete bipartite graph $G'$. (b) The complete bipartite graph $G''$.}
	\label{fig:comp-bipartite-XY}
\end{figure}

\begin{lemma}
	\label{lem:butterfly-geo6}  
If $U$ and $V$ are as above, then 
$$\gcover(M'_{U,W}(\BF(r)), U\cup W, \BF(r)) \geq \lceil(2/3)2^r\rceil\,.$$  		
\end{lemma}

\begin{proof}  
Set $G = \BF(r)$ and let $G'$ and $G''$ be the complete bipartite graphs as defined above. To prove the lemma we claim that the following holds: 
\begin{equation*} 
	\begin{split}
		\gcover(M'_{U,W}(G), U\cup W, G) 
		& \geq  \gcover(M(G''), X\cup Y, G'') \\
		& = \gcover(M(G'), X\cup Y, G') \\
		& = \gcover(G') \\
		& \geq \lceil(2/3)2^r\rceil\,. \\		
	\end{split}
\end{equation*}

In order to prove the inequality $\gcover(M'_{U,W}(G), U\cup W, G) \geq  \gcover(M(G''), X\cup Y, G'')$, we define an 1-1 mapping $f\colon M'_{U,W}(G) \to M(G'')$. This mapping $f:P \mapsto f(P)$ is defined as follows. 
\begin{enumerate}
	\item 
	Each $P'(u^b_i,u^b_j)$ of $M'_{U,W}(G)$, where $u^b_i, u^b_j \in U^b$, is mapped to geodesic $x^b_{i}y_{0}x^b_{j} \in M(G'')$, where $x^b_i, x^b_j \in X^b$.
	\item
	Each $P'(u^r_i,u^r_j)$ of $M'_{U,W}(G)$, where $u^r_i, u^r_j \in U^r$, is mapped to geodesic $x^r_{i}y_{0}x^r_{j} \in M(G'')$, where $x^r_i, x^r_j \in X^r$.
	\item 
	Each $P'(w^b_i,w^b_j)$ of $M'_{U,W}(G)$, where  $w^b_i, w^b_j \in W^b$, is mapped to geodesic $y^b_{i}x_{0}y^b_{j} \in M(G'')$, where  $y^b_i, y^b_j \in Y^b$.
	\item
	Each $P'(w^r_i,w^r_j)$ of $M'_{U,W}(G)$, where $w^r_i, w^r_j \in W^r$, is mapped to geodesic $y^r_{i}x_{0}y^r_{j} \in M(G'')$, where $y^r_i, y^r_j \in Y^r$.
	\item 
	Each $P_{w_{k}}(u^b_{i}, u^r_{j})$ of $M'_{U,W}(G)$, where $u^b_{i} \in U^b, u^r_{j} \in U^r, w_{k} \in W$, is mapped to geodesic $x^{b}_{i}{y_{k}}x^{r}_{j} \in M(G'')$ where $x^b_{i} \in X^b, x^r_{j} \in X^r, y_{k} \in Y$.
	\item
	Each $P_{u_{k}}(w^b_{i}, w^r_{j})$ of $M'_{U,W}(G)$, where $w^b_{i} \in W^b, w^r_{j} \in W^r, u_{k} \in U$, is mapped to geodesic $y^{b}_{i}{x_{k}}y^{r}_{j} \in  M(G'')$ where $y^b_{i} \in Y^b, y^r_{j} \in Y^r, x_{k} \in X$.
\end{enumerate}
If $S$ is a subset set of $M'_{U,W}(G)$ in $\BF(r)$, then let $f(S)$ be a subset of $M(G'')$ defined by $f(S)  = \{f(P) : P \in S\}$. By the mapping defined above, if the geodesics of $S$ cover $U \cup W$ of $BF(r)$, then the geodesics of $f(S)$ cover $X\cup Y$ of $G''$. See Fig.~\ref{fig:Butterfly-bipartite-XY}. Since $|S| = |f(S)|$, by applying Lemma~\ref{lem:gcover-max-geo}, we get the inequality \begin{align*}\gcover(M'_{U,W}(G), U\cup W, G) \geq  \gcover(M(G''),  X\cup Y, G'').\end{align*} 

\begin{figure}[ht!]
	\centering
	\includegraphics[scale=0.45]{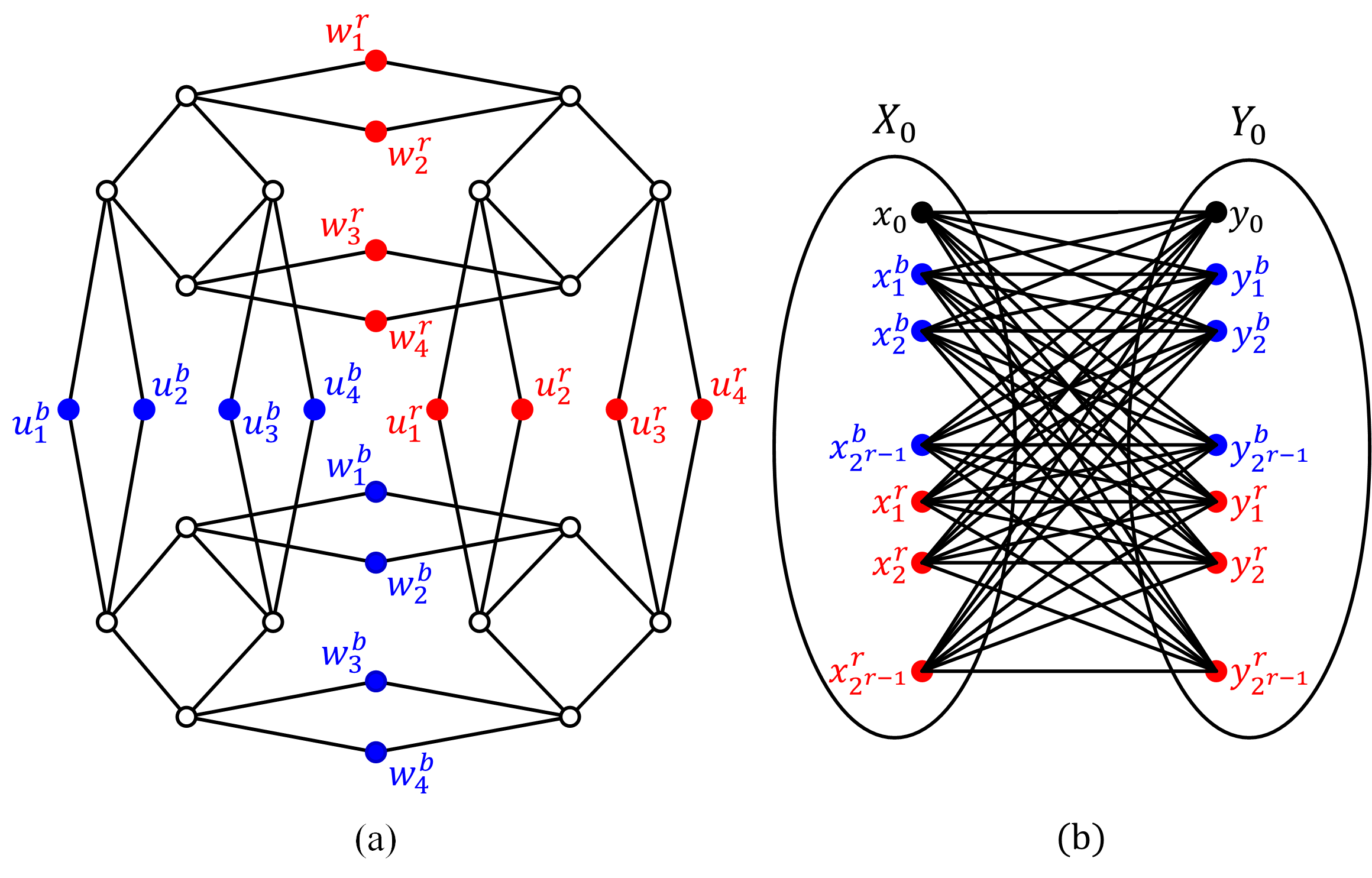}
	\caption{(a) $G = \BF(3)$.  (b) Complete bipartite graph $G''$. If a set $S$ of geodesics of $BF(r)$ cover $U \cup W$ of $BF(r)$, then the geodesics of $f(S)$ cover $X\cup Y$ of $G''$.}
	\label{fig:Butterfly-bipartite-XY}
\end{figure}

Next we shall prove that $\gcover(M(G''), X \cup Y, G'') =  \gcover(M(G'), X \cup Y, G')$. By Lemma~\ref{lem:gcover-max-geo}, we get the inequality \begin{align*}\gcover(M(G''), X \cup Y, G'') \leq  \gcover(M(G'), X \cup Y, G')\end{align*} because $M(G'')$ is superset of $M(G')$. 
Now we prove the reverse inequality. If  $P \in M(G'')$ and $V(P) \subseteq X \cup Y$, then $P \in M(G')$. In other words, if a subset $S$  of $M(G'')$ covers $X \cup Y$, there exists a subset $S'$ of $M(G')$ such that $S'$ covers $X \cup Y$ (in case $x_{0} \in S'$, then $x_0$ is replaced with any other vertex of $X_{0}-S'$) and $|S| = |S'|$.
Thus, $\gcover(M(G'), X \cup Y, G') \leq \gcover(M(G''), X \cup Y, G'')$.

Since $G'(U,V,E')$ is a complete bipartite graph $K_{2^r,2^r}$, by Lemma~\ref{lem:gcover-max-geo} and Proposition~\ref{prop:gcover-com-bipartite}, we infer that $\gcover(M(G'), X \cup Y, G') = \gcover(G') \geq \lceil(2/3)2^r\rceil$.
\end{proof}

Combining Lemma~\ref{lem:butterfly-geo4} with Lemma~\ref{lem:butterfly-geo6}, we have: 

\begin{lemma} 
	\label{lem:butter-gcover-inequ2}
	If $r\ge 2$, then $\gcover(\BF(r)) \geq \lceil(2/3)2^r\rceil$. 
\end{lemma}

\subsection{An upper bound for the geodesic cover number of butterfly networks}
\label{subsec:upperbound-gcover-butterfly}

In this section, our aim is to  construct a geodesic cover of cardinality  $\lceil (2/3)2^{r}\rceil$ for $\BF(r)$. 
In $\BF(r)$, there are $2^r$ rows and $r+1$ levels. 
The set of vertices at level $0$ is $U= \{u_1,\ldots,u_{2^r}\}$, the set of vertices at level $r$ in $W = \{w_1,\ldots,w_{2^r}\}$. In this section, the order of vertices in $U$ and $W$ are with respect to normal representation of $\BF(r)$. (In the previous section, the order was with respect to diamond representation.) Refer to Fig.~\ref{fig:Butterfly-4D-nor-9.1}. 
The set $U$ is further partitioned into $A$ and $B$, and $W$ is partitioned into $C$ and $D$, see~Fig.~\ref{fig:Butterfly-4D-nor-9.1}. These sets are formally defined as follows:
\begin{align*}
A & = \{[0,1], [0,2],\ldots, [0,2^{r-1}]\}, \\
B & = \{[0,2^{r-1}+1], [0,2^{r-1}+2], \ldots, [0,2^{r}]\}, \\
C & = \{[r,1], [r,2],\ldots, [r,2^{r-1}]\}, \\
D & = \{[r,2^{r-1}+1], [r,2^{r-1}+2], \ldots, [r,2^{r}]\}.
\end{align*}
The next important step is to color the vertices of $\BF(r)$ in two colors---red and blue. In $U$, the vertex $[0,i]$ is colored in red if $i$ is even, and is colored in blue otherwise. In $W$, the vertices of $C$ are colored in red and the vertices of $D$ are colored in blue. See~Fig.~\ref{fig:Butterfly-4D-nor-9.1} again.

We concentrate only on diametrals of $\BF(r)$ because we shall a construct geodesic cover of $\BF(r)$ in terms of diametrals. Thus, it is necessary to study the properties of diametrals of $\BF(r)$. Throughout this section, $P_{v}(u,w)$ denotes a diametral in $\BF(r)$ such that $u$ and $w$ are the end vertices of $P_{v}(u,w)$ and $v$ is the middle vertex of $P_{v}(u,w)$. Now onward, we only consider $\BF(r)$ with colored vertices as described before, see Fig.~\ref{fig:Butterfly-diamond-normal}.

\begin{property}
	\label{pro:property1}
	If a vertex $v$ is at level $0$ $($level $r)$ and vertices $u,w$ at level $r$ $($level $0)$ are in opposite colors, then there exists a unique diametral $P_{v}(u,w)$ in $\BF(r)$.  
\end{property}

\begin{proof}
	The structural details of two different representations of $\BF(r)$ which are illustrated in Fig.~\ref{fig:Butterfly-diamond-normal} are explained in \cite{Manuel2008}. By Lemma \ref{lem:max-geo-butterfly}, a geodesic $P(x,y)$ between a vertex $x$ at level $0$ and a vertex $y$ at level $r$ is unique in $\BF(r)$ and the length of $P(x,y)$ is $r$.
	From the diamond representation of $\BF(r)$ in Fig.~\ref{fig:Butterfly-diamond-normal} (b),  whenever the vertex $v$ at level $0$ $($level $r)$ is in any color and vertices $u,w$  at level $r$ $($level $0)$ are in opposite colors, there exists a diametral  $P_{v}(u,w)$ of $\BF(r)$ between $u$ and $w$ passing through $v$.   Since $P(u,v)$ and $P(v,w)$ are unique, $P_{v}(u,w)$ is also unique.
\end{proof}

\begin{figure}[ht!]
	\centering
	\includegraphics[scale=0.33]{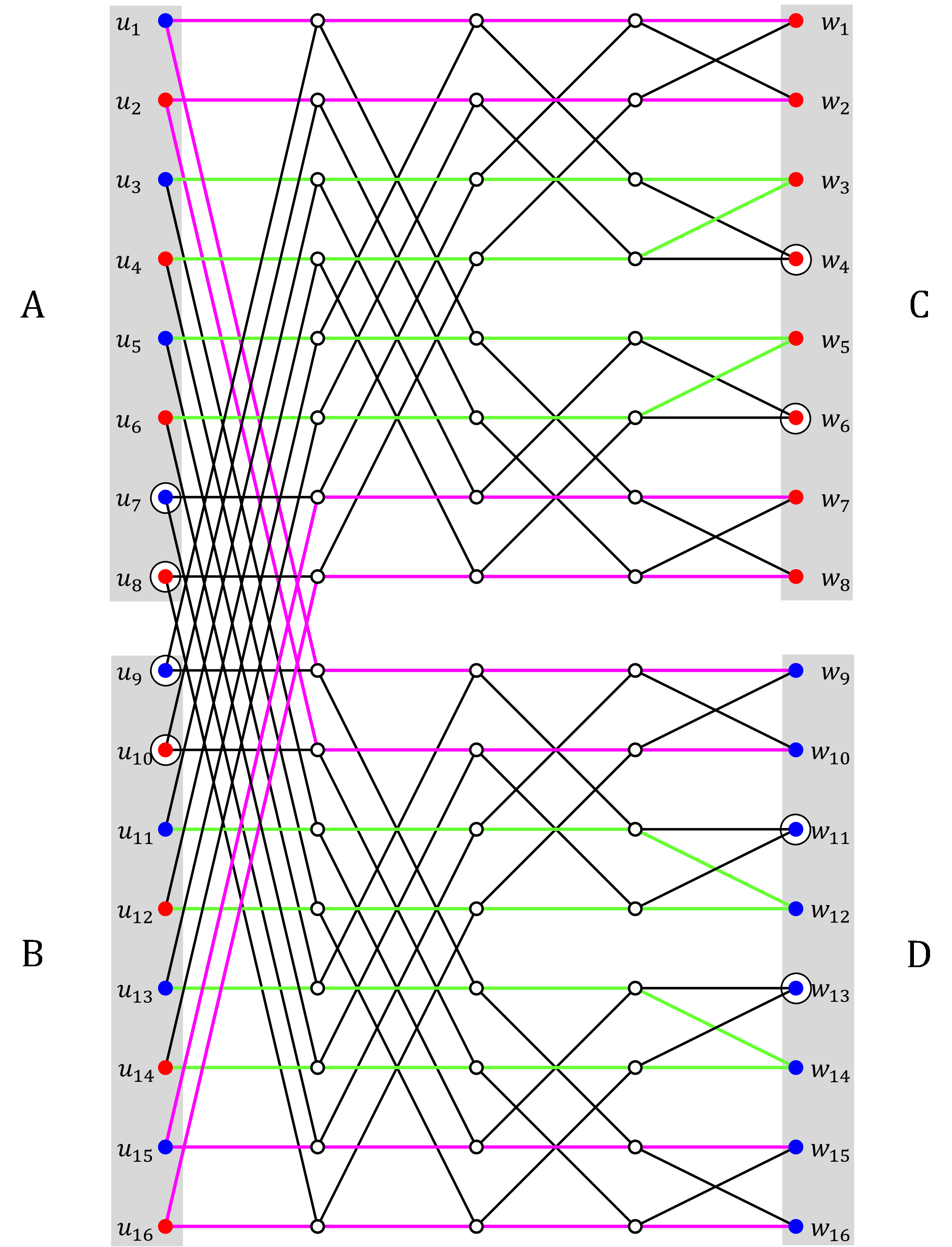}
	\caption{
In Stage $1$, geodesics $P_{u_{1}}(w_{1},w_{9})$, $P_{u_{2}}(w_{2},w_{10})$ and 
$P_{u_{16}}(w_{16},w_{8})$, $P_{u_{15}}(w_{15},w_{7})$ are constructed (pink).   
In Stage $2$, geodesics  $P_{w_{3}}(u_{3},u_{4})$, $P_{w_{5}}(u_{5},u_{6})$ and 
$P_{w_{14}}(u_{14},u_{13})$, $P_{w_{12}}(u_{12},u_{11})$  are constructed (green).
		The vertices which are not covered in previous stages are circled. They are then covered in Stage $3$.}
	\label{fig:Butterfly-4D-nor-9.1}
\end{figure}

Thus, by Property \ref{pro:property1}, in order to construct a diametral path $P_{v}(u,w)$ in $\BF(r)$, it is enough to identify the middle vertex $v$ at level $0$ (level $r$) in any color and end vertices $\{u,w\}$ at level $r$ (level $0$) in opposite colors.

The construction of a geodesic cover of $\BF(r)$ is carried out in three stages, cf.\ Fig.~\ref{fig:Butterfly-4D-nor-9.1}.

\medskip\noindent
{\bf Stage 1}

\medskip\noindent
In the first stage, the following diametrals are constructed:
\begin{enumerate}
	\item 	
	$P_{u_i}(w_i,w_{i+2^{r-1}})$, where $i \in[2^{r-3}]$ and $u_i \in A$. 
	\item 
	$P_{u_i}(w_i,w_{i-2^{r-1}})$, where $i \in \{2^r, 2^{r}-1, \ldots, 7\cdot 2^{r-3}+1\}$ and $u_i \in B$.
\end{enumerate}

\medskip\noindent
{\bf Stage 2}

\medskip\noindent
In the second stage, the following diametrals are constructed:
\begin{enumerate}
	\item 
	$P_{w_i}(u_i,u_{i+1})$, where $i \in \{2^{r-3}+1, 2^{r-3}+3, \ldots, 3\cdot 2^{r-3}\}$ and $w_i \in C$. 
	\item
	$P_{w_i}(u_i,u_{i-1})$, where $i \in \{7\cdot 2^{r-3}, 7\cdot 2^{r-3}-2, \ldots, 5\cdot 2^{r-3}+1\}$ and $w_i \in D$. 
\end{enumerate}
The vertices not covered by the diametrals during the first two stages are:
\begin{enumerate}
	\item 
	$A' = \{u_i \in A:\ i \in \{3\cdot 2^{r-3}+1, 3\cdot 2^{r-3}+2, \ldots, 2^{r-1}\}\}$. 
	\item
	$B' = \{u_i \in B:\ i \in\{2^{r-1}+1, 2^{r-1}+2, \ldots, 5\cdot 2^{r-3}\}\}$.
	\item
	$C' = \{w_i \in C:\ i \in\{2^{r-3}+2, 2^{r-3}+4, \ldots, 3\cdot 2^{r-3}\}\}$.
	\item
	$D' = \{w_i \in D:\ i \in 5\cdot 2^{r-3}+1, 5\cdot 2^{r-3}+3, \ldots, 7\cdot 2^{r-3}\}\}$. 
\end{enumerate}
Note that $A'$ has equal number of red and blue vertices and that the same holds $B'$. Also, $C'$ has only red vertices, while $D'$ has only blue vertices. Refer to Fig.~\ref{fig:Butterfly-4D-nor-9.1}. There are $2^{r-3}$ red vertices in $A'\cup B'$, $2^{r-3}$ blue vertices in $A'\cup B'$, $2^{r-3}$ red vertices in $C'$, and $2^{r-3}$ blue vertices in $D'$.

After Stage 1 and Stage 2, all the internal vertices (that is, the vertices from $V(G)-(U \cup W)$) are covered. Few vertices of $U \cup W$ remains uncovered, leading to Stage 3.

\eject

\noindent
{\bf Stage 3}

\medskip\noindent
We first regroup and rename the vertices of $A'$ and $B'$ into the following sets: 
\begin{enumerate}
	\item 
	$U^{r} = \{u^{r}_i:\ i \in[2^{r-3}]\}$ - the red vertices of $A'$ and $B'$.
	\item 
	$U^{b} = \{u^{b}_i:\ i \in [2^{r-3}]\}$ - the blue vertices of $A'$ and $B'$. 
	\item 
	$W^{r} = \{w^{r}_i:\ i \in [2^{r-3}]\}$ - the vertices of $C'$.
	\item 
	$W^{b} = \{w^{b}_i:\ i \in [2^{r-3}]\}$ - the vertices of $D'$.
\end{enumerate}
(The sets $W^{r}$ and $W^{b}$ are thus obtained by renaming the vertices of $C'$ and $D'$.) 
The next step in Stage 3 is to partition the vertices of $U^{r}$, $U^{b}$, $W^{r}$, and $W^{b}$ into four subsets. For a fixed $r \geq 5$, let $\ell= \lfloor \frac{2^{r-3}}{3} \rfloor$. Then  $2^{r-3} = 3\cdot \ell + 1$ or $2^{r-3} = 3\cdot \ell + 2$. 

\medskip\noindent
{\bf Case 1}: $2^{r-3} = 3\cdot \ell + 1$.\\
Recall that $U^{r}$ contains $2^{r-3}$ red vertices. The set $U^{r}$ is further partitioned into three subsets each subset containing $\ell$ vertices and one subset containing the remaining vertex $u^{r}_{x}$ of $U^{r}$. The other three sets $U^{b}$, $W^{r}$, and $W^{b}$ are also partitioned similarly. The partition of $U^{r}$, $U^{b}$, $W^{r}$, and $W^{b}$ and their subpartitions are:
\newline
\newline
\noindent
{\red  
	\begin{minipage}[c]{0.5\textwidth}
		\centering
		\begin{tabular}{ |p{.5cm}|p{.7cm} p{.7cm} p{.4cm} p{.6cm}|}
			\hline
			\multicolumn{5}{|c|}{$U^{r}$} \\
			\hline
			$U^{r}_1$ &  $u^{r}_{1},$ &   $u^{r}_{2},$ & $\ldots,$ &  $u^{r}_{\ell}$\\
			\hline
			$U^{r}_2$ &  $u^{r}_{\ell+1},$ &   $u^{r}_{\ell+2}$, & $\ldots,$ &  $u^{r}_{2\ell}$\\
			\hline
			$U^{r}_3$ &  $u^{r}_{2\ell+1},$ &   $u^{r}_{2\ell+2}$, & $\ldots,$ &  $u^{r}_{3\ell}$\\
			\hline
			$U^{r}_4$ &  $u^{r}_{x}$ &&&\\
			\hline
		\end{tabular}
	\end{minipage}
	\begin{minipage}[c]{0.5\textwidth}
		\centering
		\begin{tabular}{ |p{.5cm}|p{.7cm} p{.7cm} p{.4cm} p{.6cm}|}
			\hline
			\multicolumn{5}{|c|}{$W^{r}$} \\
			\hline
			$W^{r}_1$ &  $w^{r}_{1},$ &   $w^{r}_{2}$, & $\ldots$, &  $w^{r}_{\ell}$\\
			\hline
			$W^{r}_2$ &  $w^{r}_{\ell+1},$ &   $w^{r}_{\ell+2}$, & $\ldots$, &  $w^{r}_{2\ell}$\\
			\hline
			$W^{r}_3$ &  $w^{r}_{2\ell+1},$ &   $w^{r}_{2\ell+2}$, & $\ldots$, &  $w^{r}_{3\ell}$\\
			\hline
			$W^{r}_4$ &  $w^{r}_{x}$ &&&\\
			\hline
		\end{tabular}
	\end{minipage}
}
\newline
\noindent
{\blue
	\begin{minipage}[c]{0.5\textwidth}
		\centering
		\begin{tabular}{ |p{.5cm}|p{.7cm} p{.7cm} p{.4cm} p{.6cm}|}
			\hline
			\multicolumn{5}{|c|}{$U^{b}$} \\
			\hline
			$U^{b}_1$ &  $u^{b}_{1},$ &   $u^{b}_{2}$, & $\ldots$, &  $u^{b}_{\ell}$\\ [0.5ex] 
			\hline
			$U^{b}_2$ &  $u^{b}_{\ell+1},$ &   $u^{b}_{\ell+2}$, & $\ldots$, &  $u^{b}_{2\ell}$\\
			\hline
			$U^{b}_3$ &  $u^{b}_{2\ell+1},$ &   $u^{b}_{2\ell+2}$, & $\ldots$, &  $u^{b}_{3\ell}$\\
			\hline
			$U^{b}_4$ &  $u^{b}_{x}$ &&&\\
			\hline
		\end{tabular}
	\end{minipage}
	\begin{minipage}[c]{0.5\textwidth}
		\centering
		\begin{tabular}{ |p{.5cm}|p{.7cm} p{.7cm} p{.4cm} p{.6cm}|}
			\hline
			\multicolumn{5}{|c|}{$W^{b}$} \\
			\hline
			$W^{b}_1$ &  $w^{b}_{1},$ &   $w^{b}_{2}$, & $\ldots$, &  $w^{b}_{\ell}$\\
			\hline
			$W^{b}_2$ &  $w^{b}_{\ell+1},$ &   $w^{b}_{\ell+2}$, & $\ldots$, &  $w^{b}_{2\ell}$\\
			\hline
			$W^{b}_3$ &  $w^{b}_{2\ell+1},$ &   $w^{b}_{2\ell+2}$, & $\ldots$, &  $w^{b}_{3\ell}$\\
			\hline
			$W^{b}_4$ &  $w^{b}_{x}$ &&&\\
			\hline
		\end{tabular}
	\end{minipage}
}
\newline
\newline
\noindent
The motivation to partition $U^{r}$, $U^{b}$, $W^{r}$, and $W^{b}$ into three subsets of equal cardinality $\ell$ is illustrated in Fig.~\ref{fig:motivation-base-6x6}. 

\begin{figure}[ht!]
	\centering
	\includegraphics[scale=0.5]{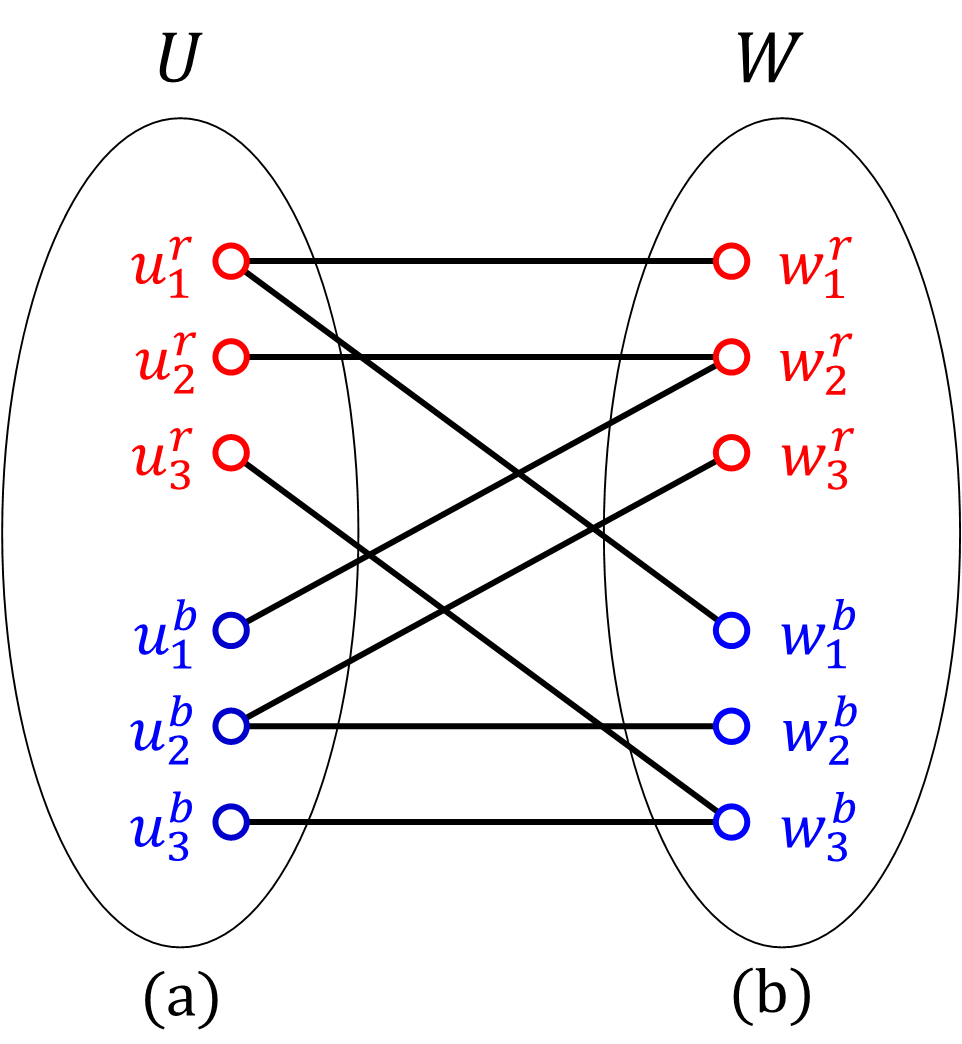}
	\caption{Here $U = \{u^{r}_{1}, u^{r}_{2}, u^{r}_{3}, u^{b}_{1}, u^{b}_{2}, u^{b}_{3}\}$ and $W = \{w^{r}_{1}, w^{r}_{2}, w^{r}_{3}, w^{b}_{1}, w^{b}_{2}, w^{b}_{3}\}$, where $u^{r}_{1}, u^{r}_{2}, u^{r}_{3}, w^{r}_{1}, w^{r}_{2}, w^{r}_{3}$ are red vertices, and $u^{b}_{1}, u^{b}_{2}, u^{b}_{3}, w^{b}_{1}, w^{b}_{2}, w^{b}_{3}$ are blue vertices. How to cover the vertices by geodesics of length $3$ with end vertices in opposite colors?}
	\label{fig:motivation-base-6x6}
\end{figure}

Using the technique of Fig.~\ref{fig:motivation-base-6x6}, the geodesics are formally constructed as follows:
\begin{enumerate}
	\item $\{P_{u^{r}_{i}}(w^{r}_{i},w^{b}_{i}):\  u^{r}_{i} \in U^{r}_{1},   w^{r}_{i} \in W^{r}_{1}, w^{b}_{i} \in W^{b}_{1}, i \in [\ell]\}$.
	\item $\{P_{w^{r}_{\ell +i}}(u^{r}_{\ell +i},u^{b}_{i}) :  w^{r}_{\ell +i} \in W^{r}_{2},   u^{r}_{\ell +i} \in U^{r}_{2}, u^{b}_{i} \in U^{b}_{1}, i [\in \ell]\}$.
	\item $\{P_{w^{b}_{2\ell +i}}(u^{b}_{2\ell +i},u^{r}_{2\ell +i}) : w^{b}_{2\ell +i} \in W^{b}_{3},   u^{b}_{2\ell +i} \in U^{b}_{3}, u^{r}_{2\ell +i} \in U^{r}_{3}, i \in [\ell]\}$.
	\item $\{P_{u^{b}_{\ell +i}}(w^{b}_{\ell +i},w^{r}_{2\ell +i}) :  u^{b}_{\ell +i} \in U^{b}_{2},   w^{b}_{\ell +i} \in W^{b}_{2}, w^{r}_{2\ell +i} \in W^{r}_{3}, i \in [\ell]\}$.
	\item $P_{u^{r}_{x}}(w^{r}_{x}, w^{b}_{x})$.
	\item Any geodesic covering of $u^{b}_{x}$.
\end{enumerate}
In Case $1$ of Stage $3$ we have thus constructed $4\ell +2$ $= \lceil \frac{2^{r-1}}{3}\rceil$ geodesics covering all the vertices of $U^{r} \cup U^{b} \cup W^{r} \cup W^{b} = A' \cup B' \cup C' \cup D'$.

\medskip\noindent
{\bf Case 2}:  $2^{r-3} = 3\cdot \ell + 2$.\\
The sets $U^{r}_{1}$, $U^{r}_{2}$, $U^{r}_{3}$ of $U^{r}$,  $U^{b}_{1}$, $U^{b}_{2}$, $U^{b}_{3}$ of $U^{b}$, $W^{r}_{1}$, $W^{r}_{2}$, $W^{r}_{3}$ of $W^{r}$, and $W^{b}_{1}$, $W^{b}_{2}$, $W^{b}_{3}$ of $W^{b}$ are the same as in Case 1. The only changes are that $U^{r}_{4} = \{u^{r}_{x}, u^{r}_{y}\}$, $U^{b}_{4} = \{u^{b}_{x}, u^{b}_{y}\}$, $W^{r}_{4} = \{w^{r}_{x}, w^{b}_{y}\}$, $W^{r}_{4} = \{w^{b}_{x}, w^{b}_{y}\}$.
The partition of $U^{r}$, $U^{b}$, $W^{r}$, and $W^{b}$ and their subpartitions are now:
\newline
\newline
\noindent
{\red  
	\begin{minipage}[c]{0.5\textwidth}
		\centering
		\begin{tabular}{ |p{.5cm}|p{.7cm} p{.7cm} p{.4cm} p{.6cm}|}
			\hline
			\multicolumn{5}{|c|}{$U^{r}$} \\
			\hline
			$U^{r}_1$ &  $u^{r}_{1},$ &   $u^{r}_{2}$, & $\ldots$, &  $u^{r}_{\ell}$\\
			\hline
			$U^{r}_2$ &  $u^{r}_{\ell+1},$ &   $u^{r}_{\ell+2}$, & $\ldots$, &  $u^{r}_{2\ell}$\\
			\hline
			$U^{r}_3$ &  $u^{r}_{2\ell+1},$ &   $u^{r}_{2\ell+2}$, & $\ldots$, &  $u^{r}_{3\ell}$\\
			\hline
			$U^{r}_4$ &  $u^{r}_{x}$, &$u^{r}_{y}$&&\\
			\hline
		\end{tabular}
	\end{minipage}
	\begin{minipage}[c]{0.5\textwidth}
		\centering
		\begin{tabular}{ |p{.5cm}|p{.7cm} p{.7cm} p{.4cm} p{.6cm}|}
			\hline
			\multicolumn{5}{|c|}{$W^{r}$} \\
			\hline
			$W^{r}_1$ &  $w^{r}_{1},$ &   $w^{r}_{2}$, & $\ldots$, &  $w^{r}_{\ell}$\\
			\hline
			$W^{r}_2$ &  $w^{r}_{\ell+1},$ &   $w^{r}_{\ell+2}$, & $\ldots$, &  $w^{r}_{2\ell}$\\
			\hline
			$W^{r}_3$ &  $w^{r}_{2\ell+1},$ &   $w^{r}_{2\ell+2}$, & $\ldots$, &  $w^{r}_{3\ell}$\\
			\hline
			$W^{r}_4$ &  $w^{r}_{x}$, &$w^{r}_{y}$&&\\
			\hline
		\end{tabular}
	\end{minipage}
}
\newline
\noindent
{\blue
	\begin{minipage}[c]{0.5\textwidth}
		\centering
		\begin{tabular}{ |p{.5cm}|p{.7cm} p{.7cm} p{.4cm} p{.6cm}|}
			\hline
			\multicolumn{5}{|c|}{$U^{b}$} \\
			\hline
			$U^{b}_1$ &  $u^{b}_{1},$ &   $u^{b}_{2}$, & $\ldots$, &  $u^{b}_{\ell}$\\ [0.5ex] 
			\hline
			$U^{b}_2$ &  $u^{b}_{\ell+1},$ &   $u^{b}_{\ell+2}$, & $\ldots$, &  $u^{b}_{2\ell}$\\
			\hline
			$U^{b}_3$ &  $u^{b}_{2\ell+1},$ &   $u^{b}_{2\ell+2}$, & $\ldots$, &  $u^{b}_{3\ell}$\\
			\hline
			$U^{b}_4$ &  $u^{b}_{x}$, &$u^{b}_{y}$&&\\
			\hline
		\end{tabular}
	\end{minipage}
	\begin{minipage}[c]{0.5\textwidth}
		\centering
		\begin{tabular}{ |p{.5cm}|p{.7cm} p{.7cm} p{.4cm} p{.6cm}|}
			\hline
			\multicolumn{5}{|c|}{$W^{b}$} \\
			\hline
			$W^{b}_1$ &  $w^{b}_{1},$ &   $w^{b}_{2}$, & $\ldots$, &  $w^{b}_{\ell}$\\
			\hline
			$W^{b}_2$ &  $w^{b}_{\ell+1},$ &   $w^{b}_{\ell+2}$, & $\ldots$, &  $w^{b}_{2\ell}$\\
			\hline
			$W^{b}_3$ &  $w^{b}_{2\ell+1},$ &   $w^{b}_{2\ell+2}$, & $\ldots$, &  $w^{b}_{3\ell}$\\
			\hline
			$W^{b}_4$ &  $w^{b}_{x}$, &$w^{b}_{y}$&&\\
			\hline
		\end{tabular}
	\end{minipage}
}
\newline
\newline
\noindent
As in Case $1$, we we can construct $4\ell +3$ geodesics to cover all the vertices of $U^{r}$, $U^{b}$, $W^{r}$, and $W^{b}$. We have thus constructed $4\ell +3$ $= \lceil \frac{2^{r-1}}{3}\rceil$ geodesics covering all the vertices of $U^{r} \cup U^{b} \cup W^{r} \cup W^{b} = A' \cup B' \cup C' \cup D'$.

Stage $1$ constructs $2^{r-3} + 2^{r-3} = 2^{r-2}$ geodesics,  Stage $2$ constructs $2^{r-3} + 2^{r-3} = 2^{r-2}$ geodesics, and Stage $3$ constructs $\lceil \frac{2^{r-1}}{3}\rceil$ geodesics, in total $2^{r-2} + 2^{r-2} + \lceil \frac{2^{r-1}}{3}\rceil$ =  $\lceil (2/3) 2^{r}\rceil$ geodesics. Together with Lemma~\ref{lem:butter-gcover-inequ2} this gives our main result: 

\begin{theorem}
	\label{gcoverBF(5)}
If $r \geq 5$, then $\gcover(\BF(r)) = \lceil (2/3) 2^{r}\rceil$.
\end{theorem}

In Theorem~\ref{gcoverBF(5)} we require $r \geq 5$ because $\ell= \lfloor \frac{2^{r-3}}{3} \rfloor$ is well-defined only when $r \geq 5$. It can be checked by hand that $\gcover(\BF(1))=2$, $\gcover(\BF(2))=4$, $\gcover(\BF(3))=6$ and $\gcover(\BF(4))=12.$

\section{The edge geodesic cover problem}
\label{sec:edge-butterfly}

In this section we turn our attention to the edge geodesic cover problem for $\BF(r)$. An edge $uv$ of $\BF(r)$ is called a {\it $(2,4)$-edge} if $\{\deg(u), \deg(v)\} = \{2, 4\}$. The number of $(2,4)$-edges of $\BF(r)$ is $2^{r+2}$~\cite{Leighton1992, Manuel2008, RaMaPa16}. 

\begin{lemma}
	\label{lem:upper-bnd-gparte}
If $r\ge 3$, then $E(\BF(r))$ can be partitioned by a set $S(r)$ of edge-disjoint isometric cycles of length $4r$, where $|S(r)| = 2^{r-1}$ and each isometric cycle of $S(r)$ has two vertices at level $0$.
\end{lemma}

\begin{proof}
	The proof is by induction on the dimension $r$ of $\BF(r)$. The base case is $r=3$ and is ellaborated in Fig.~\ref{fig:Butterfly-edge-gparte-base}, where $E(\BF(3))$ is partitioned by a set $S(3)$ of edge-disjoint isometric cycles of length $4\cdot 3$ and each isometric cycle of $S(3)$ has two vertices at level $0$. 
	
\begin{figure}[ht!]
	\centering
	\includegraphics[scale=0.55]{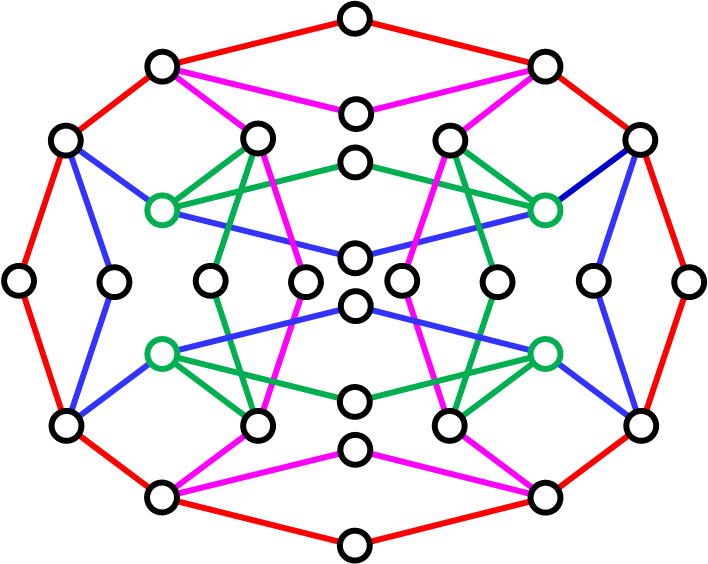}
	\caption{The base case is $\BF(3)$ in which $E(\BF(3))$ is partitioned by a set $S$ of edge-disjoint isometric cycles of length $4\cdot 3$ where $|S| = 2^{3-1} = 4$.}
	\label{fig:Butterfly-edge-gparte-base}
\end{figure}
	
Assume now that the edge set of $\BF(k-1)$ can be partitioned by a set $S(k-1)$ of edge-disjoint isometric cycles of length $4(k-1)$, where $|S| = 2^{k-2}$ and each isometric cycle $C$ of $S(k-1)$ has two vertices $u$ and $w$ at level $0$. A cycle $C$ in $S(k-1)$ is represented by $u-P-w-Q-u$ where $u$ and $w$ are the two vertices of $C$ at level $0$, $P$ is the path segment of length $2k-2$ in $C$ between $u$ and $w$ and $Q$ is the path segment of length $2k-2$ in $C$ between $u$ and $w$. Refer to Fig.~\ref{fig:upperbound-gparte}. 

\begin{figure}[ht!]
	\centering
	\includegraphics[scale=0.40]{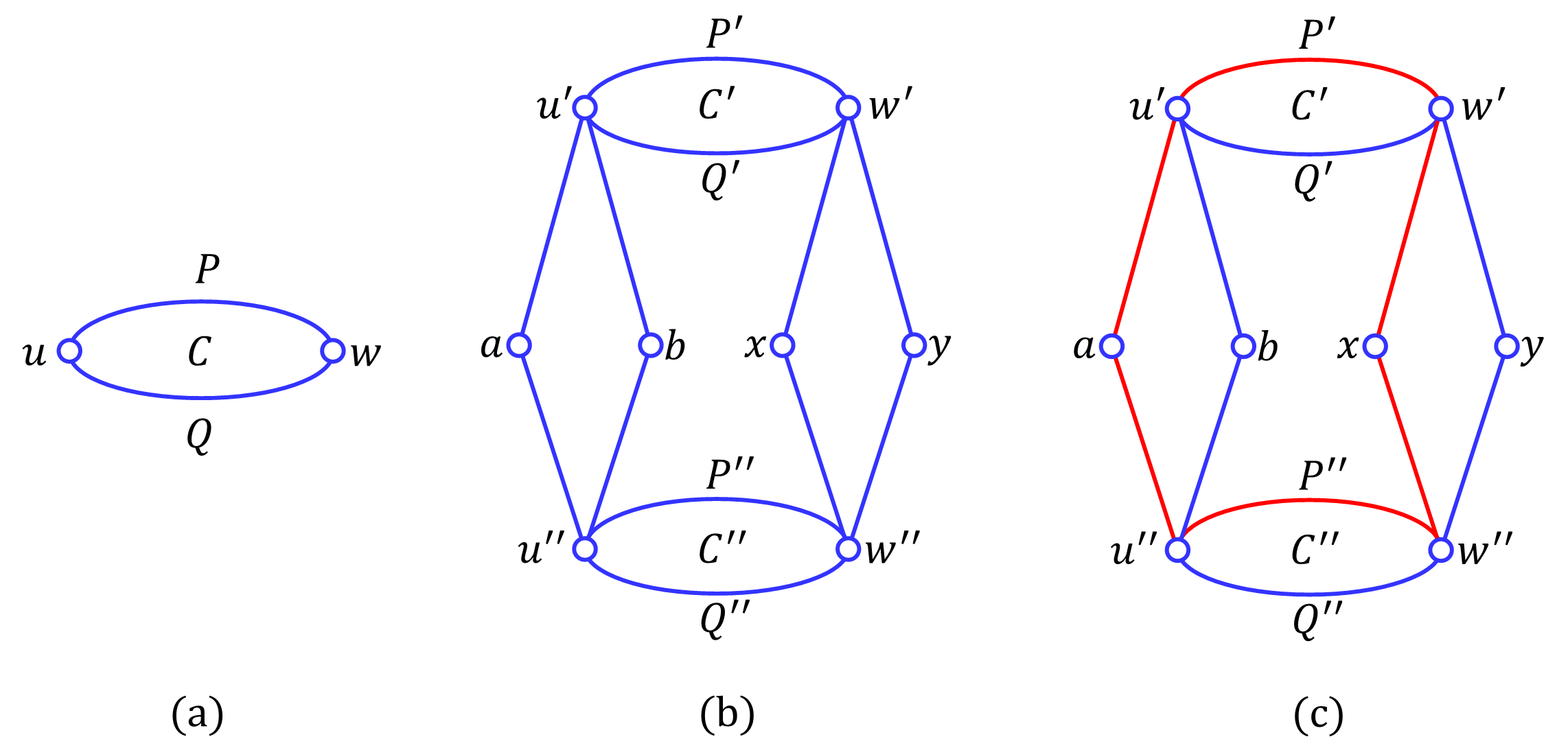}
	\caption{(a) An isometric cycle $C=u-P-w-Q-u$ in $\BF(k-1)$. (b) Two copies of $C$ are $C'$ and $C''$. The two isometric cycles $C'$ and $C''$ are connected by vertices $a$, $b$, $x$ and $y$ which are at level $0$. (c) The two isometric cycles $C'$ and $C''$  generate two different isometric cycles $C_1$ and $C_2$ (marked in different color) in $\BF(k)$.}
	\label{fig:upperbound-gparte}
\end{figure}

Recall that $\BF(r)$ has two copies of $\BF(k-1)$, $\BF'(k-1)$ and $\BF''(k-1)$, where  $V(\BF'(k-1)) = \{v':\ v \in \BF(k-1)\}$ and $V(\BF''(k-1)) = \{v'':\ v \in \BF(k-1)\}$. Also, there are two copies of $S(k-1)$ in $\BF(k)$, $S'(k-1)$ and $S''(k-1)$, where 
\[S'(k-1) = \{C' = u'-P'-w'-Q'-u':\ C = u-P-w-Q-u \in S(k-1)\}\] and \[S''(k-1) = \{C'' = u''-P''-w''-Q''-u'':\ C = u-P-w-Q-u \in S(k-1)\}.\] The length of cycles $C \in S(k-1)$, $C' \in S'(k-1)$, and $C'' \in S''(k-1)$ is $4(k-1)$.
	In $\BF(r)$, there are two vertices $a$ and $b$ at level $0$ which are adjacent to $u'$ of $C'$ and $u''$ of $C''$. In the same way, there are two vertices $x$ and $y$ at level $0$ which are adjacent to $w'$ of $C'$ and $w''$ of $C''$.  Refer to Fig.~\ref{fig:upperbound-gparte}. Now define the cycles \[C_1 = a-u'-P'-w'-x-w''-P''-u''-a\] and \[C_2 = b-u'-Q'-w'-y-w''-Q''-u''-b.\] Since, $P'$, $P''$, $Q'$, and $Q''$ each are of length $2k-2$, it is easy to observe that the length of $C_1$ and $C_2$ is $4k$. Let us define $S(k) = \{C_1, C_2 : C \in S(k-1)\}$. Each cycle $C_1$ and $C_2$ have two vertices at level $0$. It is easy to observe that the cycles in $S(k)$ are isometric and they are mutually edge-disjoint. Moreover, the cardinality of $S(k)$ is $2^{k-1}$. Thus, the edge set of $\BF(k)$ is partitioned by the edge-disjoint isometric cycles of $S(k)$ such that $|S(k)| = 2^{k-1}$.
\end{proof}

\begin{theorem}
\label{thm:ButterflyGcovere}
If $r\ge 3$, then $\gcovere(\BF(r)) = \gparte(\BF(r)) = 2^r$.
\end{theorem}

\begin{proof}
	In order to derive the claimed lower bound for $\gparte(\BF(r))$, let us consider all $(2,4)$-edges of $\BF(r)$. As already mentioned, there are $2^{r+2}$ $(2,4)$-edges in $\BF(r)$. Since a geodesic can cover a maximum of four $(2,4)$-edges of $\BF(r)$, $\gcovere(\BF(r)) \geq 2^{r+2}/4 = 2^r$. Thus, \[\gparte(\BF(r)) \geq \gcovere(\BF(r)) \geq 2^r.\]

To prove $\gparte(\BF(r)) \leq 2^r$, it is enough to construct an edge geodesic partition of cardinality $2^r$ for $\BF(r)$.   By Lemma \ref{lem:upper-bnd-gparte},  the edge set of $\BF(r)$ can be partitioned by a set $S$ of edge-disjoint isometric cycles of length $4r$, where $|S| = 2^{r-1}$.  In other words, the edge set of $\BF(r)$ can be partitioned by a set $R$ of edge-disjoint diametrals of length $2r$ such that $|R| = 2^{r}$. Thus, $\gparte(\BF(r)) \leq 2^r$.
\end{proof}

\section{Conclusion}
\label{sec:conclusion}
The geodesic cover problem is one of the fundamental problems in graph theory, but only partial solutions are available for most situations. The geodesic cover number in both vertex and edge version was unknown for butterfly networks, in this paper we provide a complete solutions for both versions. 

Even though the geodesic cover and geodesic partition are frequently used in fixed interconnection networks, the exact values of geodesic cover number and geodesic partition
number are unknown for popular architectures such as shuffle-exchange, de Bruijn, Kautz, star, pancake, circulant, wrapped butterfly, CCC networks. The geodesic cover problem and the geodesic partition problem (their edge versions) are wide open for researcher.

\section*{Acknowledgments}
This work was supported and funded by Kuwait University, Research Grant No.\ (FI 01/22).


\baselineskip14pt
\begin{thebibliography}{99}

\bibitem{ApCaSi04}
N.~Apollonio, L.~Caccetta, B.~Simeone, 
Cardinality constrained path covering problems in grid graphs, 
Networks 44 (2004) 120--131.

\bibitem{ChDa22}
D.~Chakraborty, A.~Dailly, S.~Das, F.~Foucaud, H.~Gahlawat, and S.~Ghosh, Complexity and algorithms for isometric path cover on chordal graphs and beyond.
In 33rd International Symposium on Algorithms and Computation (ISAAC 2022). Leibniz International Proceedings in Informatics (LIPIcs), Volume 248, pp. 12:1-12:17, Schloss Dagstuhl – Leibniz-Zentrum f\"ur Informatik (2022).


\bibitem{Fish2001}
D.C.~Fisher, S.L.~Fitzpatrick, 
The isometric number of a graph, 
J.\ Combin.\ Math.\ Combin.\ Comput.\ 38 (2001) 97--110.

\bibitem{Fitz1999}
S.L.~Fitzpatrick,
The isometric path number of the Cartesian product of paths,
Congr.\ Numer.\ 137 (1999) 109--119.

\bibitem{Fitz2001}
S.L.~Fitzpatrick, R.J.~Nowakowski, D.A.~Holton, I.~Caines, 
Covering hypercubes by isometric paths, 
Discrete Math.\ 240 (2001) 253--260. 

\bibitem{HsuLin08}
L.H.~Hsu, C.K.~Lin,
Graph Theory and Interconnection Networks,  
CRC Press,  2008. 
  
\bibitem{Leighton1992}
F.T.~Leighton, 
Introduction to Parallel Algorithms and Architectures. Arrays, Trees, Hypercubes,
Morgan Kaufmann, 1992.

\bibitem{Manuel18}
P.~Manuel,
Revisiting path-type covering and partitioning problems,
arXiv:1807.10613 [math.CO] (25 Jul 2018). 

\bibitem{Manuel19}
P.~Manuel,
On the isometric path partition problem,
Discuss.\ Math.\ Graph Theory 41 (2021) 1077--1089.

\bibitem{Manuel2008}
P.~Manuel, I.M.~Abd-El-Barr, I.~Rajasingh, B.~Rajan,
An efficient representation of {B}enes networks and its applications,
J.\ Discrete Alg.\ 6 (2008) 11--19.

\bibitem{PaCh05}
J.J.~Pan, G.J.~Chang, 
Isometric-path numbers of block graphs, 
Inf.\ Process.\ Lett.\ 93 (2005) 99--102. 

\bibitem{pan-2006}
J.-J.~Pan, G.J.~Chang, 
Isometric path numbers of graphs,
Discrete Math.\ 306 (2006) 2091--2096. 

\bibitem{RaMaPa16} 
I.~Rajasingh, P.~Manuel, N.~Parthiban, D.A.~Jemilet, R.S.~Rajan, 
Transmission in butterfly networks, 
Comput.\ J.\ 59 (2016) 1174--1179.

\bibitem{LiSa22}
C.V.G.C.~Lima, V.F.~dos Santos, J.H.G.~Sousa, S.A.~Urrutia, 
On the computational complexity of the strong geodetic recognition problem,
RAIRO Oper.\ Res.\ 58 (2024) 3755--3770.

\bibitem{SuRaRaRa-2021}
J.~Sujana G., T.M.~Rajalaxmi, I.~Rajasingh, R.S.~Rajan, 
Edge forcing in butterfly networks,
Fund.\ Inform.\ 182 (2021) 285--299. 

\bibitem{ToDa05}
A.~Touzene, K.~Day, B.~Monien, 
Edge-disjoint spanning trees for the generalized butterfly networks and their applications, 
J.\ Parallel Distrib.\ Comput.\ 65 (2005) 1384--1396.
\end{thebibliography}
\end{document}